\documentclass[12pt]{amsart}
\usepackage[text={6in,8.5in},centering]{geometry}
\usepackage{eepic}
\usepackage{epic}
\usepackage[]{graphics}
\usepackage{amsmath, amssymb}
\usepackage{verbatim}               

\newtheorem{theorem}{Theorem}[section]
\newtheorem{lemma}[theorem]{Lemma}

\theoremstyle{definition}

\theoremstyle{remark}
\newtheorem{remark}[theorem]{Remark}

\usepackage{bm}

\numberwithin{equation}{section}

\def\bbeta{\ensuremath{\bm{\beta}}}

\newcommand{\undern}{{\bf n}}

\newcommand{\underv}{\mathbf{v}}
\newcommand{\bb}{\mathbf{b}}
\newcommand{\ds}{\displaystyle}

\newcommand{\e}{{\mbox{\tiny E}}}
\newcommand{\E}{{\mbox{\sc E}}}
\newcommand{\dedge}[1]{\delta_\e #1}
\newcommand{\harmbeta}{\mathcal{H}_\e(\bbeta)}

\newcommand{\Nedelec}{N\`{e}d\`{e}lec }
\newcommand{\Reals}[1]{\ensuremath{\rm I\! R}^{#1}}

\begin{document}

\title[Exponential fitting scheme for tensor coefficients]
{An exponential fitting scheme for general convection-diffusion equations on tetrahedral meshes}
\thanks{The authors thank 
the Center of Applied Scientific Computing at Livermore National Laboratory,  University of California, 
for the partial financial support for this work through contract B529214.
}
 
\author{R.~D. Lazarov}
\address{Department of Mathematics,
Texas A \& M University,
College Station, TX 77843, U.S.A.}
\email{lazarov@math.tamu.edu}

\author{L.~T. Zikatanov}
\address{Department of Mathematics,
Pennsylvania State University,
University Park, PA 16802, U.S.A.}
\email{ludmil@psu.edu}

\subjclass{65N30, 65N15}

\date{A version of this note appeared in Computational and Applied
  Mathematics, (Obchysljuval'na ta prykladna matematyka, Kiev) Vol. 1,
  Number 92, pp. 60-69, 2005.}

\begin{abstract}
This paper contains construction and analysis a finite element
approximation for convection dominated diffusion problems with full
coefficient matrix on general simplicial partitions in $\Reals{d}$, $d \geq 2$.  This construction is quite close to the scheme of
Xu and Zikatanov \cite{zikatanov_xu_99} where a diagonal coefficient
matrix has been considered. The scheme is of the class of
exponentially fitted methods that does not use upwind or checking the
flow direction. It is stable for sufficiently small discretization
step-size assuming that the boundary value problem for the
convection-diffusion equation is uniquely solvable. Further, it is
shown that, under certain conditions on the mesh the scheme is
monotone.  Convergence of first order is derived under minimal
smoothness of the solution.
\end{abstract}

\dedicatory{Dedicated to Academician Alexander Andreevich
Samarskii -- a pioneer in numerical analysis, computational mathematics and 
computational physics on the occasion of his 85th birthday}
\maketitle

%%%%%%%%%%%%%%%%%%%%%%%%%%%%%%%%%%%%%%%%%%%%%%%%%%%%%%%%%%%%%%%%%%%%%%%%%%%

\section{Introduction}\label{section:intro}

%%%%%%%%%%%%%%%%%%%%%%%%%%%%%%%%%%%%%%%%%%%%%%%%%%%%%%%%%%%%%%%%%%%%%%%%%%%

We consider the following convection-diffusion-reaction problem: 
Find $u=u(x)$ such that
\begin{equation} \label{basic-problem}
\left \{
 \begin{array}{rll}
   L u \equiv  - \nabla\cdot (D\nabla u + \bb u)+\gamma u
                         &= f &\mbox{ in } ~\Omega, \\[1.5ex]
                       u &= 0            &\mbox{ on } ~\Gamma_D, \\[1.5ex]
     -D (\nabla u +\bb u) \cdot \undern
                         &= g &\mbox{ on} ~\Gamma_N^{in},\\[1.5ex]
       D\nabla u \cdot \undern 
                         &= 0&\mbox{ on} ~\Gamma_N^{out}.
\end{array}
\right.
\end{equation}
Here $\Omega$ is a bounded polygonal domain in $\Reals{d}$, $d=2,3$,
$D=D(x)$ is $d \times d$ symmetric, bounded and uniformly positive
definite matrix in $\Omega$, $\bb^t=(b_1(x), \dots, b_d(x)) $ is a given vector function,
$\undern $ is the unit outer vector normal to $\partial \Omega$,
and $f$ is a given source function. We have also used the notation
$\nabla u$ for the gradient of a scalar function $u$ and $\nabla \cdot \bb$
for the divergence of a vector function $\bb$ in $\Reals{d}$. The boundary of $\Omega$,
$\partial \Omega$, is split into Dirichlet, $\Gamma_D$, and Neumann,
$\Gamma_N$, parts. Further, the Neumann boundary is divided into
two parts: $\Gamma_N =\Gamma_N^{in} \cup \Gamma_N^{out}$, where
$\Gamma_N^{in}=\{ x \in \Gamma_N: \undern(x) \cdot \bb(x) > 0 \}$ and
$\Gamma_N^{out}=\{ x \in \Gamma_N: \undern(x) \cdot \bb(x) \leq 0 \}$.
We assume that $\Gamma_D$ has positive surface measure. 
The  case $D(x)=\epsilon I$,
where $I$ is the identity matrix in $\Reals{d}$  and $\epsilon > 0$ is a small
parameter, corresponds to the important and difficult class of 
isotropic singularly perturbed convection-diffusion problems.

Various generalizations have wide practical applications.
For example, $\gamma u$ could be replaced by nonlinear reaction term $\gamma(u)$ or
the linear convective flux $\bb u$ could be replaced by a nonlinear advection
flux $\bb(u)$. Finally, $u$ could be a vector function describing the concentration
of various chemicals or biological components so that (\ref{basic-problem})
is a system of equations coupled through the absorption/reaction term.
Now $\gamma $ is a matrix 
that models the chemical reactions or the biological interaction of the components.
All these cases give rise to mathematical problems of convection dominated 
processes with possibly anisotropic diffusion.

Our study of numerical method for solving (\ref{basic-problem}) is motivated by the fact 
that the above problem  
is the simplest model of transport and dispersion of a passive contaminant in 
porous media. If the pressure $p(x)$ in the aquifer is known 
(or already has been computed by solving a standard diffusion problem)
then the pressure gradient forces the ground water to flow.
%scaled by the permeability $A(x)$
%pressure gradient, $\underv=  A(x) \nabla p$, forces the ground water to flow.
The transport of a contaminant dissolved in the water,
is described by the dispersion-reaction equation (\ref{basic-problem}), 
where $u(x)$ represents the contaminant concentration, 
%$\bb$ is the Darcy velocity  
$\bb = \underv =  A \nabla p$ is the Darcy velocity
(up to a sign), $A$ is the  permeability of the porous media, 
$\gamma $ is the bio-degradation/absorption 
rate, and $D(x)$ is the diffusion-dispersion matrix given by
\begin{equation}\label{D-projection}
  D(x) = k_{d} I + k_t\bb \bb^t/|\bb| + k_l( |\bb|I -\bb \bb^t/|\bb|). 
%\underv \underv^t/|\underv| + k_l( |\underv|I -\underv \underv^t/|\underv|).
\end{equation}
Here $k_{d}$, $k_t$, and $k_l$ are coefficients of diffusion, transverse 
dispersion, and longitudinal dispersions, respectively (cf. \cite{dagan}). 
In dispersive underground flows $k_t > k_l$ which implies that $D(x)$ is
positive definite matrix, but possibly ill-conditioned. This problem exhibits 
all difficulties associated with this class: monotone solutions
that are highly localized due to internal and boundary layers, 
material heterogeneities and orthotropy, complex geometry, etc.
%boundary and internal layers, material heterogeneities and orthotropy, etc.

Among the deficiencies of the standard finite element, finite volume,
and finite difference approximations are loss of monotonicity, so that
the numerical solution often exhibits non-physical oscillations, loss
of solvability of the resulting algebraic problem, poor local
resolution, fast dissipation of the energy, etc.  A.~A.~Samarskii was
one of the first to encounter the difficulties that arise in the
numerical solution of such problems.  In the early 60-ies
A.~A.~Samarskii addressed most of the issues for one-dimensional
problems that resulted in a new scheme described in his monograph
\cite[Chapter 4]{samarskii_book}.

In the past 40 years many special approximation techniques have been
developed for multidimensional problems, for structured and
unstructured grids, for general second order elliptic operators, etc.
These techniques include monotone and upwind finite difference, finite
volume, and finite element methods (e.g. ~ \cite{RBank_DRose_1987},
\cite{carstensen_lazarov_tomov},
\cite{LDurlofsky_BEngquist_SOsher_1992}, \cite{samarskii_book},
\cite{sam_monotone}, \cite{sam_transport}, \cite{sam_triangular}, and
\cite{tabata_77}),
streamline diffusion stabilization of the finite element method
(e.g. ~ %\cite{brezzi_russo_94},
\cite{RBank_JBuerger_WFichtner_RSmith_1990}, \cite{brezzi_russo_94},
\cite{ABrooks_THughes_1982}, \cite{THughes_1995}, and
\cite{CJohnson_1987a}), and special functional spaces setting
(e.g. ~\cite{canuto_tabaco_01} and \cite{sangalli_04}).  For more
information regarding numerical methods and analytical techniques in
solving and studying convection-diffusion equations, especially
convection dominated problems, we refer to the monograph of Ross,
Stynes and Tobiska \cite{HRoos_MStynes_LTobiska_1996}.
 
On a continuous level many convection-diffusion satisfy maximum
principle.  This is a desirable property of the solution of the
resulting discrete problem as well. Scheme that satisfies maximum
principle is often called {\it monotone scheme}. Among the several
aforementioned schemes, upwind schemes are often monotone provided
that the coefficient matrix $D(x)$ is diagonal.  In the recent works
\cite{sam_monotone}, \cite{sam_transport} Samarskii and his co-workers
were able to derive monotone schemes on rectangular meshes for the
problem (\ref{basic-problem}) when $D(x)$ is a full matrix.  These
schemes are second order accurate on uniform meshes and solution in
$C^3$.
 
The idea of construction of monotone schemes for singularly perturbed 
convection-diffusion problem goes back to the
work by Scharfetter and Gummel \cite{S-Gummel}, where the monotonicity
%method in numerical semiconductor device modeling 
has been a very desired property in numerical semiconductor device
modeling.  Exponentially fitted scheme for a general
convection-diffusion problem with diagonal matrix $D(x)$ on an
arbitrary simplicial mesh was derived and studied
in~\cite{PMarkowich_MZlamal_1989a},
\cite{FBrezzi_LMarini_PPietra_1989a},
\cite{FBrezzi_LMarini_PPietra_1989b}, \cite{LMarini_PPietra_1990},
\cite{JJHMiller_SWang_CWu_1988}, \cite{JJHMiller_SWang_1991}. As we
mentioned, the technique we use here is a generalization of
\cite{zikatanov_xu_99} for general meshes and is also related to 
the methods proposed in \cite{Ludmil_1992} (see also~\cite{Ludmil_1991} 
for simple diode simulation using such scheme).
Under mild conditions on the mesh (the partition has to satisfy
certain angle condition) it has been shown that the scheme is
monotone. Further, in \cite{zikatanov_xu_99} it was proved that the
scheme converges with first order provided that the solution $u \in
W^{1,p}$ and the flux $D (x)\nabla u + \bb(x) u \in (W^{1,p})^d$ for $
p>d$.  Note, that these are very mild conditions on smoothness of the
solution of problem (\ref{basic-problem}).

The aim of this note is to construct an exponentially fitted finite
element approximation of (\ref{basic-problem}) on general simplicial 
partitions, for symmetric positive definite matrices $D(x)$, and for arbitrary 
vector-functions $\bb$. 
The proposed scheme is a generalization of the discretization 
derived in \cite{zikatanov_xu_99}, \cite{Ludmil_1992} for problems with 
diagonal matrices $D(x)$.
Important role in the construction and the analysis 
plays the expansion of a constant over each element vector-flux using the 
lowest order \Nedelec{} basis for simplicial finite elements. This allows to 
present the bilinear form
through differences of the vertex values of the test functions
and the exponentially weighted local solution. This  
representation ensures the consistence of the method.

The scheme has several interesting
features. It is a finite element scheme with a standard variational
formulation (but with a modified bilinear form); 
it does not use explicitly the standard upwind techniques, such
as checking the flow directions; it can be applied to very general
unstructured grid in any spatial dimension. 
It would be difficult to expect that in such generality 
the scheme will be monotone. Nevertheless, %for monotonicity of the scheme
we were able to find conditions that involve the geometry of the finite 
elements in a metric associated with $D$ so that the scheme is monotone. Further, for
sufficiently small step-size of the finite element partition we prove 
existence and uniqueness of the solution of the discrete problem 
by using a fundamental result of Schatz \cite{ASchatz_1974}. 

The paper is organized as follows. In Section \ref{section:notation} we introduce the 
necessary notation for Sobolev spaces, finite element partition and the
discrete space. Section \ref{section:scheme} contains the main results of the paper. 
In Subsection \ref{section:prelim} we present the rationale used in 
derivation of exponentially fitting finite element scheme.
An important concept here is the edge based interpolation of the total flux that 
uses an ordinary differential equation along the edge. In Subsection
\ref{section:discretization} we present the scheme itself as a consequence of this
special interpolation. The main result here is contained in Lemma \ref{lemma:auxiliary} 
where certain properties of the discrete bilinear form are obtained. Finally, in
Subsection \ref{section:well_posed} we prove the stability of the scheme for sufficiently small 
mesh-size  and derive an  estimate for the error under minimal smoothness of the solution.

%%%%%%%%%%%%%%%%%%%%%%%%%%%%%%%%%%%%%%%%%%%%%%%%%%
%%%%%%%%%%%%%%%%%%%%%%%%%%%%%%%%%%%%%%%%%%%%%%%%%%
 
%\section{Preliminaries}\label{prelim}
\section{Notation}\label{section:notation}
 
%%%%%%%%%%%%%%%%%%%%%%%%%%%%%%%%%%%%%%%%%%%%%%%%%%
%%%%%%%%%%%%%%%%%%%%%%%%%%%%%%%%%%%%%%%%%%%%%%%%%%

In this section, we introduce the necessary notation and describe some
basic properties of finite element partitions and finite element spaces. 
 
We denote by $L^p(K)$, $1\leq p \leq \infty$ the space of
$p$-integrable real-valued functions
over $K \subset \Omega$ (with the usual modification for $p=\infty$), 
by $(\cdot,\cdot)_K$ and $~|| \cdot ||_K$, respectively, the inner product
and the norm in $L^2(K)$. Further $|\cdot|_{1,p,K}$ and $||\cdot||_{1,p,K}$, respectively
denote the semi-norm and norm of the Sobolev space $W^{1,p}(K)$. For $p=2$ we use 
$ H^1(K) := W^{1,2}(K)$ and if $K=\Omega$  often we suppress the index $K$ so that
$(\cdot, \cdot)_{\Omega}:=(\cdot,\cdot)$ and $~\| \cdot \|_{\Omega} :=\|\cdot\|$,
and $~\| \cdot \|_{1, \Omega} :=\|\cdot\|_1$.
Further, we use the Hilbert space
$$H_D^1(\Omega) = \{ v \in H^1(\Omega): ~v|_{\Gamma_D} = 0\}.$$
%Finally, we denote by $H^{1/2}(\partial K)$ the space of the traces of
%functions in $H^1(K)$ on the boundary $\partial K$.

We introduce the bilinear form $a(\cdot, \cdot)$ defined on
$H_D^1(\Omega) \times H_D^1(\Omega)$:
\begin{equation}\label{bilinear-form}
   a(u, v)  :=
     (D \nabla u + \bb u, \nabla v) + (\gamma u, v) -
             \int_{\Gamma_N^{out}} \bb \cdot \undern ~u ~v \, d s .
\end{equation}
Then (\ref{basic-problem}) has the following weak form:
Find $ u \in H_D^1(\Omega)$ such that
\begin{equation}\label{equation:weak}
   a(u, v)=F(v) := (f,v) + \int_{\Gamma_N^{in}}gv\, ds\quad\mbox{for all}
    ~~v \in H_D^1(\Omega).
\end{equation}
Further in the paper we assume that the following inf-sup condition is valid:
there is a constant $c_0 > 0$, such that
\begin{equation}\label{inf-sup}
% \alpha_0 \| u \|_{1} \leq 
\sup_{v \in H^1_D(\Omega)} \frac{a(w,v)}{\|v\|_1} \geq c_0 \| w \|_{1}, 
           \quad \forall w \in H^1_D(\Omega).
\end{equation}
We shall also assume that the bilinear form $ a(w, v)$ is bounded on  $H_D^1(\Omega)$
and the linear form $F(v)$ is continuous in $H_D^1(\Omega)$.
Then the above problem has unique solution (cf. \cite{Gilbarg_Trudinger}). 

\begin{remark}
A sufficient condition for (\ref{inf-sup}) and continuity of $ a(u, v)$ and 
$F(v)$ are, for example, 
$ \gamma(x) + 0.5 \nabla \cdot \bb(x) \geq 0$ for all
$x \in \Omega$, boundedness of the coefficients $D(x)$, $\bb(x)$, and 
$\gamma(x)$ in $\Omega$.
\end{remark}

Let $\mathcal{T}_h$ be a family of simplicial finite element
triangulations of $\Omega$ that are shape regular and satisfy the
usual conditions (see \cite[Chapter 2]{PCiarlet_1978}).  For simplicity of
the exposition, we assume that the triangulation covers $\Omega$ exactly.
Associated with each $\mathcal{T}_h$, let $V_h \subset H^1_D(\Omega)$ be
the finite element space of piece-wise linear functions. 
By $v_I \in V_h$ we denote the standard
finite element Lagrange interpolant which assumes the values of 
%a given smooth  function 
$v \in C^0$ at  the vertexes in the partition  $\mathcal{T}_h$.
 
Given $T\in \mathcal{T}_h$, we introduce the following notation. By $q_j$, 
$j=1,\dots,4 $ we denote the vertices of $T$, $\E$ is
the edge connecting two vertices $q_i$ and $q_j$,
%(the indexes $i,j$ will be clear from the context), 
 $\dedge{\phi} = \phi(q_i)-\phi(q_j)$ for any continuous function
$\phi$ on $\E$, and  
$\tau_\e=\dedge \, x = q_i-q_j$ is a directional vector of $\E$ (not assumed unitary).

%%%%%%%%%%%%%%%%%%%%%%%%%%%%%%%%%%%%%%%%%%%%%%%%%%%%%%%%%%%%%%%%%%%%%%%%%%
%%%%%%%%%%%%%%%%%%%%%%%%%%%%%%%%%%%%%%%%%%%%%%%%%%%%%%%%%%%%%%%%%%%%%%%%%%

\section{Exponential fitting scheme for general convection-diffusion equations}
\label{section:scheme}

%%%%%%%%%%%%%%%%%%%%%%%%%%%%%%%%%%%%%%%%%%%%%%%%%%%%%%%%%%%%%%%%%%%%%%%%%%
%%%%%%%%%%%%%%%%%%%%%%%%%%%%%%%%%%%%%%%%%%%%%%%%%%%%%%%%%%%%%%%%%%%%%%%%%%

\subsection{Preliminaries}\label{section:prelim}

%%%%%%%%%%%%%%%%%%%%%%%%%%%%%%%%%%%%%%%%%%%%%%%%%%%%%%%%%%%%%%%%%%%%%%%%%%

Introduce a notation for the scaled flux
\begin{equation}\label{flux}
J(u)= \nabla u + \bbeta(x) u, \qquad \bbeta=D^{-1} \bb.
\end{equation}
We will assume further that $J(u)\in [W^{1,p}(\Omega)]^d$,  $p>n$,
$D, ~D^{-1} \in [W^{1,\infty}(\Omega)]^{d\times d}$ and $\bb\in
W^{1,\infty}(\Omega)$. These assumptions on the coefficients smoothness can be
relaxed to hold element-wise (i.e. for each $T\in \mathcal{T}_h$) and
the considerations below will still hold with changes of some of the
norms used in the error estimate to be taken element-wise as well.  

The basic idea which we use in the construction of the exponentially
fitted scheme is to approximate the flux vector $J(u)$ with a constant
vector field $J_T(u)$ on each element $T$ of the partition
$\mathcal{T}_h$.
Apparently, if $J_T(u)$ is a constant on each simplex, then we can
expand it
using the \Nedelec{} basis as follows:
$$
J_T=\sum_{E\in T} J_T\cdot \tau_\e \, \varphi_\e(x).  
$$
Here $\varphi_\e$ are the \Nedelec{} basis functions, which in terms
of the barycentric coordinates $\lambda_i$ are given by
$$
\varphi_\e:=\lambda_i\nabla\lambda_j-
\lambda_j\nabla\lambda_i, \quad \E=(q_i, q_j).
$$
The goal then is to write out $J_T(u)\cdot \tau_\e$ in
terms of $u$, for all edges $\E$ and thus determine the
approximation.  
To find the moments of the tangential flux, we
use the same technique as in \cite{zikatanov_xu_99}.
Let $u \in H^1_0(\Omega) \cap
C^0(\bar{\Omega})$.  Consider an edge $\e\subset\partial T$. Taking the Euclidean inner product with $\tau_\e$ we obtain
$$
\ds (\nabla u \cdot \tau_\e) +
(\bbeta \cdot\tau_\e) u =
(J(u) \cdot \tau_\e).
$$
A change of variables in this ordinary differential
equation then gives:
\begin{equation}\label{prebasic}
\ds e^{-\psi_\e} \partial_\e (e^{\psi_\e} u) =
%\frac{\partial (e^{\psi_\e} u)}{\partial \tau_\e} =
\frac{1}{|\tau_\e|}(J(u) \cdot \tau_\e), \quad \mbox{where}\quad
\partial_\e \psi_\e =
%\frac{\partial\psi_\e}{\partial\tau_\e} =
\frac{1}{|\tau_\e|}(\bbeta\cdot\tau_\e)
\end{equation}
and $\partial_\e v:= \nabla v \cdot \tau_\e/|\tau_\e|$ is the 
directional derivative along the edge $\E$.
After integration over $\E$ we obtain that
\begin{equation*}
\ds \dedge(e^{\psi_\e}u) =
\frac{1}{|\tau_\e|}
\int_\e
e^{\psi_\e}
(J(u)\cdot\tau_\e) ds.
\end{equation*}
Let $\harmbeta$ be the harmonic average of $e^{\psi_\e}$
over $\E$ defined as follows:
\begin{equation}\label{harm}
\ds \harmbeta = \left[\frac{1}{|\tau_\e|}
\int_{\e}e^{\psi_\e} ds \right]^{-1}.
\end{equation}
The constant approximation $J_T$ is then obtained by using the mean
value theorem $J^*\cdot\tau_\e\int_\e e^{\psi_\e}\;ds = \int_\e
J\cdot\tau_\e e^{\psi_\e}\;ds$, and the definition then is
\begin{equation*}
\ds J_T(u)\cdot\tau_\e :=
J^*\cdot\tau_\e=
\harmbeta \dedge(e^{\psi_\e}u).
\end{equation*}

%%%%%%%%%%%%%%%%%%%%%%%%%%%%%%%%%%%%%%%%%%%%%%%%%%%%%%%%%%%%%%

\subsection{Discrete problem}\label{section:discretization}

%%%%%%%%%%%%%%%%%%%%%%%%%%%%%%%%%%%%%%%%%%%%%%%%%%%%%%%%%%%%%%

We now have all the ingredients needed to define the discrete
approximation to \eqref{equation:weak}.  Based on the above
considerations, we shall define two approximate bilinear forms. The first one
is used in the formulation of the discrete problem and the second is
used in an intermediate step needed to prove the error estimate.

On a fixed element $T\subset \mathcal{T}_h$, we first introduce
\begin{equation}\label{bilineara}
\ds a_{h,T}(u_h,v_h) = \sum_{\e\subset\partial T} \omega^T_\e(D) \harmbeta
\dedge(e^{\psi_\e}u_h) \dedge v_h,
\end{equation}
where
$$
\omega^T_\e(D) = - \int_T D\nabla\lambda_i\cdot\nabla \lambda_j\;dx,
\quad \E=(q_i,q_j).
$$
Note, that $\omega^T_\e(D)$ give the element stiffness matrix for 
the diffusion part of the differential equation, $-\nabla \cdot (D(x) \nabla u)$.

Next, we use the expansion via the \Nedelec{} basis functions, to
define
\begin{equation}\label{bilinearb}
\ds b_{h,T}(u_h,v_h) = \sum_{\e\subset\partial T} \harmbeta
\dedge(e^{\psi_\e}u_h)
\int_T D\varphi_\e\cdot \nabla v_h\;dx.
\end{equation}
The global bilinear form is then obtained by summing over all elements
of the triangulation the local forms (\ref{bilineara}) and adding the 
contributions from the boundary $\Gamma_N^{out}$, as follows:
\begin{equation}\label{bilinear}
\ds a_{h}(u_h,v_h) =  
\sum_{T\in\mathcal{T}_h} a_{h,T}(u_h,v_h) + \int_\Omega \gamma u_h v_h dx - 
     \sum_{E \subset \Gamma_N^{out}} \int _{E} \bb \cdot \undern \, u_h v_h ds.
\end{equation}
Finally, the finite element approximation of the problem \eqref{basic-problem}
reads as follows: Find
$u_h\in V_h$ such that
\begin{equation}\label{equation:discrete-problem}
a_h(u_h,v_h)=F(v_h), \quad \mbox{for all} \quad v_h\in V_h.
\end{equation}

The following lemma is the main tool used in the analysis of the above scheme.
%error analysis.
%
\begin{lemma}\label{lemma:auxiliary} 
The following relations hold for any $v_h\in V_h$:
\begin{enumerate}
\item[1.] If $w \in C(\overline T)$ then
\begin{equation}\label{item 1}
b_{h,T}(w_I,v_h) =
\ds
\sum_{\e\subset\partial T}
\left[\frac{\harmbeta}{|\tau_\e|}
\int_{\e}{e^{\psi_\e}}
J(w)\cdot \tau_\e ds\right] \int_T D\varphi_\e\cdot
\nabla v_h;
\end{equation}
\item[2.] If $J_T$ is a constant vector on $T$, then for any $v_h\in V_h$
\begin{equation}\label{item 2}
\sum_{\e\subset\partial T} J_T\cdot \tau_\e\int_T D \varphi_\e\cdot\nabla v_h\;dx =
\sum_{\e\subset\partial T} \omega^T_\e(D) J_T \cdot \tau_\e \dedge v_h;
\end{equation}
\item[3.] If $w\in C(\overline T)$ and $J(w)\in [W^{1,p}(T)]^n$, $p>d$, then the
following inequality holds for every $v_h \in V_h$ and $T \in {\mathcal T}_h$
\begin{equation}\label{item 3}
|a_T(w,v_h)-a_{h,T}(w_I,v_h)| \leq
C h  %\left\{
|J(w)|_{1,p,T} %^2\right\}^{1/2}
\|v_h\|_{1, T}.
\end{equation}
where 
$$
a_T (w, v_h)=\int_T (D \nabla w + \bb w) \cdot \nabla v_h \, dx.
$$
%is the bilinear form from \eqref{equation:weak}, restricted to $T$.
\end{enumerate}
\end{lemma}
\begin{proof}
The proof of~1. follows directly from the derivation.

The proof of 2. can be done as follows: Consider $\Phi:=J_T\cdot x$
for $x \in T$ (here $x$ is the vector of coordinates in
$\Reals{d}$). It is obvious that $\Phi$ is linear and that
$\nabla\Phi = J_T$ and hence $\nabla\Phi\cdot \tau_\e = J_T\cdot
\tau_\e$.  
We now use the fact that the \Nedelec{} canonical interpolation 
$\Pi_{\mathcal{N}} J_T = \sum_{\e\in T} (J_T \cdot\tau_\e) \varphi_\e$
satisfies the commutativity property $\Pi_{\mathcal{N}}\nabla\Phi =
\nabla \Phi_I = \nabla \Phi$.  Therefore,
\begin{eqnarray*}
\sum_{\e\subset\partial T} J_T\cdot \tau_\e\int_T D \varphi_\e\cdot\nabla
v_h\;dx &=& 
\int_T D \left[
\sum_{\e\subset\partial T} J_T\cdot \tau_\e \varphi_\e\right]\cdot\nabla
v_h\;dx \\
 & = & \int_T D [\Pi_{\mathcal{N}} J_T]\cdot\nabla v_h\;dx\\
 & = & \int_T D[\Pi_{\mathcal{N}} J_T]\cdot\nabla v_h\;dx
=  \int_T D [\Pi_{\mathcal{N}} \nabla\Phi]\cdot\nabla v_h\;dx\\
&=&  \int_T D [\nabla\Phi]\cdot\nabla v_h\;dx
= \sum_{\e\subset\partial T} \omega^T_\e(D) J_T \cdot \tau_\e \dedge v_h,
\end{eqnarray*}
and the proof of~2. is complete

To prove~3. 
we use \eqref{item 1} and split the difference in the following way
\begin{equation}
\ds a_T(w,v_h)-a_{h,T}(w_I,v_h) =\mathcal{E}_1(J(w),v_h) + \mathcal{E}_2(J(w),v_h)
\end{equation}
where
\begin{equation*}
\ds \mathcal{E}_1(J(w),v_h)=a_T(w,v_h)-b_{h,T}(w_I,v_h),
%\quad\mbox{and}\quad
%\mathcal{E}_2(J,v_h)=b_{h,T}(w,v_h)-a_{h,T}(w_I,v_h).
\end{equation*}
and 
\begin{equation*}
\mathcal{E}_2(J(w),v_h)=b_{h,T}(w_I,v_h)-a_{h,T}(w_I,v_h).
\end{equation*}
From the relations and item~1.  we can expand the forms
$a_T(w,v_h)$, $a_{h,T}(w_I,v_h)$ and $b_{h,T}(w_I,v_h)$ to get that
\begin{equation}
\begin{array}{rcl}
\ds \mathcal{E}_1(J(w),v_h)& = &
\ds \int_{T} J(w) \cdot \nabla v_h dx \\[2.5ex]
& - &
\ds \sum_{\e\subset\partial T}
\left[\frac{\harmbeta}{|\tau_\e|}
\int_{\e}{e^{\psi_\e}}
J(w)\cdot \tau_\e ds\right]\; \,
%\\[2.5ex] & & \ds \times 
\;\int_T D\varphi_E\cdot\nabla v_h\;dx
\end{array}
\end{equation}
and
\begin{equation}
\begin{array}{rcl}
\ds \mathcal{E}_2(J(w),v_h)& = &
\ds \sum_{\e\subset\partial T}
\left[\frac{\harmbeta}{|\tau_\e|}
\int_{\e}{e^{\psi_\e}}
J(w)\cdot \tau_\e ds\right] \int_T D\varphi_E\cdot\nabla v_h\;dx\\[3.5ex]
& & -
\ds \sum_{\e\subset\partial T}
\omega_\e^T(D) \left[\frac{\harmbeta}{|\tau_\e|}
\int_{\e}{e^{\psi_\e}}
J(w)\cdot \tau_\e ds\right] \dedge{v_h}.
\end{array}
\end{equation}
A change of variable from the standard reference element $\widehat T$
(of unit size)
to $T$: $x = B \hat{x} + b_0$ and $w(x)= \hat w(\hat x) $
gives the following scaled bilinear forms
\begin{eqnarray*}
a_T(w,v_h)
& = &  
|{\rm det} B|
\int_{\hat{T}} (
B^{-1} \hat{D}\widehat{J(w)}
\cdot\nabla \hat{v}_h)\;d\hat{x}\\
a_{h,T}(w,v_h)
& = &
\sum_{\hat{\e} \subset \hat{T}}
\omega_\e^T(D)
\left[\frac{\harmbeta}
{|\tau_{\hat{\e}}|}
\int_{\hat{\e}}{e^{\widehat{\psi_\e}}}
(\widehat{J(w)}\cdot \tau_{\hat{\e}})
d\hat{s}\right]\dedge{\hat{v}_h}\\
b_{T,h}(w,v_h)
& = &
|{\rm det} B|
\sum_{\hat{\e} \subset \hat{T}}
\left[\frac{\harmbeta}
{|\tau_{\hat{\e}}|}
\int_{\hat{\e}}{e^{\widehat{\psi_\e}}}
(\widehat{J(w)}\cdot \tau_{\hat{\e}})
d\hat{s}\right]\; \\
 & & \qquad \qquad \;\times \int_{\hat{T}} (
B^{-1} \hat{D}\hat{\varphi}_{\hat{\e}}
\cdot\nabla \hat{v}_h)\;d\hat{x}.
\end{eqnarray*}
By our assumptions on the smoothness of $J(w)$, the corresponding
error functionals 
$\hat{ \mathcal{E} }_i(\widehat{ J(w) },\hat{ v}_h)$, $i=1,2$,  
can be appropriately bounded:
\begin{equation}\label{equation:bound-hat}
\hat{ \mathcal{E} }_i(\widehat{ J(w) },\hat{ v }_h)
\le
%\|\hat{v}_h\|_{1,\hat{T}}
C_i \|\widehat{J(w)}\|_{0,\infty,\hat T}\|\hat{v}_h\|_{1,\hat{T}},
\end{equation}
where $C_i$ might depend on $D$, but do not depend on
$\bbeta$.
By the Sobolev inequality we have that
$$
\|\widehat{J(w)}\|_{0,\infty,\hat T} \leq C
\|\widehat{J(w)}\|_{1,p,\hat T}, ~p > d.
$$ We observe that from \eqref{item 1} and \eqref{item 2}, it follows
that $\mathcal{E}_{i}(J(w),v_h) = 0$ if $J(w)$ is a constant on $T$.
By applying the Bramble-Hilbert Lemma on $\widehat{T}$, and scaling
back to $T$ we obtain the desired result:
\begin{equation}\label{localestimate}
|\mathcal{E}_i(J(w),v_h)| \leq  C h|J(w)|_{1,p,T} |v_h|_{1,T}, \quad i=1,2.
\end{equation}
\end{proof}

%%%%%%%%%%%%%%%%%%%%%%%%%%%%%%%%%%%%%%%%%%%%%%%%%%%%%%%%%%%%%%%%%%%%

\subsection{Solvability of the discrete problem and error estimate}
\label{section:well_posed}

%%%%%%%%%%%%%%%%%%%%%%%%%%%%%%%%%%%%%%%%%%%%%%%%%%%%%%%%%%%%%%%%%%%%

In this paragraph we state two lemmas related to the solvability of
the problem and then a result related to the error bound. The first
result, the proof of which follows straightforward from the definition, is
related to the monotonicity of the scheme (i.e. discrete maximum
principle). This amounts to a condition on the geometry of the mesh
associated with the matrix $D$.
\begin{lemma}\label{lemma:discr-max-princ}
The stiffness matrix corresponding to the bilinear form
(\ref{equation:discrete-problem}) is an $M$-matrix for any continuous
function $\bbeta$ if and only if the following inequality holds for
all edges $E$ in the triangulation
\begin{equation}
\sum_{T\supset\e} \omega_\e^T  \ge 0
\end{equation}
\end{lemma}
One may check out easily that if $D$ is a constant matrix and $d=2$
(two spatial dimensions) this is equivalent to the statement that the
triangulation is Delaunay in the metric introduced by $D$. Namely, 
instead of Euclidean inner product $\bb \cdot \undern$ of the vectors 
$\bb$ and $\undern$ in  $\Reals{d}$ we need to use the 
inner product $ D \bb \cdot \undern$ (recall that $D$ is a symmetric 
and positive definite matrix).
In this case the global stiffness matrix of the finite element system is 
nonsingular and therefore the scheme 
\eqref{equation:discrete-problem} has unique solution.
 
Next result is about solvability of the discrete problem for
sufficiently small characteristic mesh size $h$.  Let us first
consider an auxiliary discrete problem with $a(\cdot,\cdot)$ in place
of $a_h(\cdot,\cdot)$ in \eqref{equation:discrete-problem}.  The
latter problem is solvable, and a convincing (but not rigorous)
argument to prove this claim is that the convection term is one order
lower than the diffusion term and hence, decreasing $h$ will make the
diffusion term dominating and the problem weakly coercive. Some more 
detailed considerations and a rigorous arguments can be found in
Schatz~\cite{ASchatz_1974} or Xu~\cite{JXu_1994a}. 
\begin{lemma}\label{lemma:infsup}
For sufficiently small $h$ the following inf-sup condition holds
\begin{equation}\label{dis-infsup}
\sup_{v_h\in V_h}\frac{a_h(w_h,v_h)}{\|v_h\|_{1,\Omega}}\ge
c_1\|w_h\|_{1,\Omega} \quad\forall w_h\in V_h
\end{equation}
with a constant $c_1>0$ independent of mesh-size $h$.
\end{lemma}
\begin{proof}
As we have pointed out, when the original bilinear form
$a(\cdot,\cdot)$ is used in \eqref{equation:discrete-problem}, the
discrete problem is uniquely solvable (for sufficiently small
$h$). Hence, there exists a constant $c_2$ such that
\begin{equation}\label{infsup3}
\sup_{v_h\in V_h}\frac{a(w_h,v_h)}{\|v_h\|_{1,\Omega}}\ge
c_2 \|w_h\|_{1,\Omega},
\quad\forall w_h\in V_h.
\end{equation}
Let $v_h, w_h \in V_h$. Then obviously
$$
\ds a_h(w_h,v_h) = a(w_h,v_h) + \left[a_h(w_h,v_h)- a(w_h,v_h)\right].
$$ 
The first term is estimated using the condition (\ref{infsup3}).  
To estimate the second term we use \eqref{item 3} from
lemma~\ref{lemma:auxiliary}, sum up over all $T$ and apply the
Schwarz inequality to obtain that
$$
|a(w_h,v_h)-a_h(w_h,v_h)| \le Ch
\left\{
\sum_{T\in\mathcal{T}_h}
|J(w_h)|^2_{1,p,T}
\right\}^{1/2} \|v_h\|_{1,\Omega}.
$$ Observing that $|w_h|_{2,T}=0$ for any $w_h \in V_h$,
$T\in\mathcal{T}_h$ we get
$$
|J(w_h)|_{1,p,T} \le C
%  \| D\|_{1,\infty,T}+
\|\bbeta \|_{1,\infty,T} \|w_h\|_{1,T}.
$$
Summing over all the elements of the partition we have
\begin{equation}
|a(w_h,v_h)-a_h(w_h,v_h)| \leq
C\,h\, \max_{T\in\mathcal{T}_h}
\|\bbeta \|_{1,\infty,T} \|w_h\|_{1,\Omega}\|v_h\|_{1,\Omega},
\end{equation}
and for $h$ satisfying 
$$
\ds h \le h_0 \equiv C \left[
\max_{T\in\mathcal{T}_h}
\| \bbeta \|_{1,\infty,T}
\right]^{-1}
$$
the discrete problem has a unique solution. 
\end{proof}

As a consequence of Lemma~\ref{lemma:auxiliary},
Lemma~\ref{lemma:infsup} we get the following convergence result.
\begin{theorem}\label{theorem:the-theorem}
Let $u$ be the solution of the problem (\ref{equation:weak}).  Assume
that for all $T \in \mathcal{T}_h$, 
$D\in\left(W^{1,\infty} (T) \right)^{d\times d}$, 
$\bbeta \in [W^{1,\infty}(T)]^d$, $u \in W^{1,p}(T)$, $\gamma \in C(\overline T)$, and 
$J(u)\equiv\nabla u + \bbeta (x) u \in \left(W^{1,p} (T) \right)^d$, $p > d$.
%$\gamma \in C(\overline T)$. %$\gamma u \in $.
Then for sufficiently small $h$, the following estimate holds:
\begin{equation}\label{estgeneral}
\ds \|u_I-u_h\|_{1,\Omega} \le
C h
\left\{
\sum_{T\in\mathcal{T}_h}
|J(u)|^2_{1,p,T}+
\sum_{T\in\mathcal{T}_h}
%|\gamma 
|u|^2_{1,p,T}
\right\}^{1/2}
\end{equation}
\end{theorem}

\begin{remark}
There are also other possibilities for expressing the flux $J(u)$. 
For example, instead of the flux  \eqref{flux} one can write
$$
\begin{array}{rcl}
D(x)\nabla u+\bb u & = & \ds D(x) \alpha(x)^{-1}
\left(\alpha(x)\nabla(u) + \alpha(x) D^{-1}(x) \bb u)\right) \\[2ex]
  &:= & \ds \widetilde D(x)\left(\alpha(x)\nabla(u) + \bbeta u)\right),
%~~\bbeta=\alpha(x)D^{-1}(x)\bb.
\end{array}
$$ 
where 
$$~\bbeta= \alpha(x) D^{-1}(x) \bb, \quad \widetilde D(x)=D(x) \alpha(x)^{-1}$$
with  $\alpha(x)$ a suitable positive function (or a positive diagonal matrix).  
Then define
$$J(u)= \alpha(x)\nabla(u) + \bbeta u.$$ 
For such a choice of $J(u)$  the derivation and the analysis of an exponentially fitting 
scheme are essentially the same with some changes occurring
in the harmonic averages used to define the discrete problem. 

For example, one may choose 
$$
\alpha(x) = (\lambda_{min}(D(x)) + \lambda_{max}(D(x))/2,
$$ 
where $\lambda_{min}(D(x))$ and $\lambda_{max}(D(x))$ are 
the minimum and maximum eigenvalues of $D(x)$. 
For such choice $\widetilde D^{-1}=\alpha D^{-1}$ is better conditioned.
For example, problem \eqref{basic-problem}, \eqref{D-projection} with data such as
$k_{d}=0.0001$, $k_t = 21$, and $k_l=2.1$, used in \cite{carstensen_lazarov_tomov},
might require such modification.
\end{remark}

\begin{remark}
As we have pointed out in the introduction, in many cases $D(x)$
takes the form  \eqref{D-projection}. Then introducing the 
orthogonal projection $\pi_{\bb}=\bb \bb^t/|\bb|^2$
along the vector $\bb(x)$ we can rewrite $D(x)$ in the form
$$
%D(x)=\alpha_1 I + \alpha_2\pi_{\bb}+\alpha_3(I-\pi_{\bb}),
D(x)=k_d I + k_t |\bb| \pi_{\bb}+ k_l |\bb| (I-\pi_{\bb}).
$$ 
%where $\pi_{\bb} := \bb \bb^t/|\bb|^2$ is an orthogonal projection along $\bb$ and
%$\alpha_i$ depend on $\underv$, $k_{d}$, $k_t$ and $k_l$.  
Now one easily finds that
$
D^{-1}\bb= ( k_d + k_t |\bb|)^{-1}\bb,
$ 
%for a suitable constant $\sigma_0$. 
%In another word, 
i.e. the evaluation of $D^{-1} \bb$ 
is just a multiplication of $\bb$ by a scalar.
%does not need inversion of the matrix $D$.
\end{remark}

%%%%%%%%%%%%%%%%%%%%%%%%%%%%%%%%%%%%%%%%%%%%%%%%%%%%%%%%%%%%%%%%%%%%%%%%%%%
\vspace{0.3in}
\subsection*{Acknowledgment}
\noindent\par
This paper is dedicated to the 85-th birthday of Academician Alexander
Andreevich Samarskii -- a pioneer in numerical analysis and
computational mathematics and computational physics. Under his
longstanding leadership the Keldysh Institute of Applied Mathematics
at the Russian Academy of Sciences and the Department of Computational
Mathematics and Cybernetics at Moscow State University had played
fundamental role in establishing {\it Mathematical Modeling} as a
dynamic branch of contemporary mathematics.  We are pleased to
acknowledge the vision, the dedication, and the seminal contributions
of Acad. A.~A. Samarskii to this important research area, which has
become a the main link between science and engineering on the basis of
mathematics and computer information technologies.

%%%%%%%%%%%%%%%%%%%%%%%%%%%%%%%%%%%%%%%%%%%%%%%%%%%%%%%%%%%%%%%%

\bibliographystyle{plain}
\bibliography{EAFE-general}

\begin{thebibliography}{10}

\bibitem{RBank_JBuerger_WFichtner_RSmith_1990}
R.~Bank, J.~B\"{u}rger, W.~Fichtner, and R.~Smith.
\newblock Some up-winding techniques for finite element approximations of
  convection diffusion equations.
\newblock {\em Numer. Math.}, 58:185--202, 1974.

\bibitem{RBank_DRose_1987}
R.~Bank and D.~Rose.
\newblock Some error estimates for the box method.
\newblock {\em SIAM J. Numer. Anal.}, 24:777--787, 1987.

\bibitem{brezzi_russo_94}
F~Brezzi and A.~Russo.
\newblock Choosing bubbles for advection-diffusion problems.
\newblock {\em Math. Models Methods Appl. Sci.}, 32(1-3):571--587, 1994.

\bibitem{FBrezzi_LMarini_PPietra_1989a}
Franco Brezzi, Luisa~Donatella Marini, and Paola Pietra.
\newblock Numerical simulation of semiconductor devices.
\newblock In {\em Proceedings of the {E}ighth {I}nternational {C}onference on
  {C}omputing {M}ethods in {A}pplied {S}ciences and {E}ngineering
  ({V}ersailles, 1987)}, volume~75, pages 493--514, 1989.

\bibitem{FBrezzi_LMarini_PPietra_1989b}
Franco Brezzi, Luisa~Donatella Marini, and Paola Pietra.
\newblock Two-dimensional exponential fitting and applications to
  drift-diffusion models.
\newblock {\em SIAM J. Numer. Anal.}, 26(6):1342--1355, 1989.

\bibitem{ABrooks_THughes_1982}
A.~Brooks and T.H. Hughes.
\newblock Streamline upwind/{P}etrov-{G}alerkin formulations for convection
  dominated flows with particular emphasis on the incompressible
  {N}avier-{S}tokes equations.
\newblock {\em Comp. Meth. in Appl. Mech. Eng.}, 32:199--259, 1982.

\bibitem{canuto_tabaco_01}
C.~Canuto and A.~Tabacco.
\newblock An anisotropic functional setting for convection-diffusion problems.
\newblock {\em East-West J. Numer. Math.}, 9(3):199--231, 2001.

\bibitem{carstensen_lazarov_tomov}
C.~Carstensen, R.~Lazarov, and S.~Tomov.
\newblock Explicit and averaging a posteriori error estimates for adaptive
  finite volume methods.
\newblock {\em SIAM J. Numer. Anal.}, 42(6):2496--2521 (electronic), 2005.

\bibitem{PCiarlet_1978}
P.~Ciarlet.
\newblock {\em The Finite Element Method for Elliptic Problems}, volume~4 of
  {\em Studies in Mathematics and its Applications}.
\newblock North-Holland Publishing Co., Amsterdam, 1978.
\newblock Studies in Mathematics and its Applications, Vol. 4.

\bibitem{dagan}
G.~Dagan.
\newblock {\em Flow and Transport in Porous Formations}.
\newblock Springer-Verlag, Berlin-Heidelberg, 1989.

\bibitem{LDurlofsky_BEngquist_SOsher_1992}
L.~Durlofsky, B.~Engquist, and S.~Osher.
\newblock Triangle based adaptive stencils for the solution of hyperbolic
  conservation laws.
\newblock {\em J. Compt. Phys.}, 98(1):199--259, 1992.

\bibitem{Gilbarg_Trudinger}
D.~Gilbarg and N.~S. Trudinger.
\newblock {\em Elliptic partial differential equations of second order}, volume
  224 of {\em Grundlehren der Mathematischen Wissenschaften [Fundamental
  Principles of Mathematical Sciences]}.
\newblock Springer-Verlag, Berlin, second edition, 1983.

\bibitem{THughes_1995}
T.H. Hughes.
\newblock {G}reens functions, the {D}irichlet-to-{N}eumann formulation, subgrid
  scale models, bubles and the origins of stabilized methods.
\newblock {\em Comp. Meth. in Appl. Mech. Eng.}, 127:387--401, 1995.

\bibitem{CJohnson_1987a}
C.~Johnson.
\newblock {\em Numerical solution of partial differential equations by the
  finite element method}.
\newblock Cambridge University Press, Cambridge, 1987.

\bibitem{LMarini_PPietra_1990}
Luisa~Donatella Marini and Paola Pietra.
\newblock New mixed finite element schemes for current continuity equations.
\newblock {\em COMPEL}, 9(4):257--268, 1990.

\bibitem{PMarkowich_MZlamal_1989a}
Peter~A. Markowich and Milo{\v{s}}~A. Zl{\'a}mal.
\newblock Inverse-average-type finite element discretizations of selfadjoint
  second-order elliptic problems.
\newblock {\em Math. Comp.}, 51(184):431--449, 1988.

\bibitem{JJHMiller_SWang_1991}
J.~J.~H. Miller and S.~Wang.
\newblock A triangular mixed finite element method for the stationary
  semiconductor device equations.
\newblock {\em RAIRO Mod\'el. Math. Anal. Num\'er.}, 25(4):441--463, 1991.

\bibitem{JJHMiller_SWang_CWu_1988}
J.~J.~H. Miller, S.~Wang, and C.~Wu.
\newblock A mixed finite element method for the stationary semiconductor device
  equations.
\newblock {\em Engineering Computations}, 5:285--288, 1988.

\bibitem{HRoos_MStynes_LTobiska_1996}
H.-O. Ross, M.~Stynes, and L.~Tobiska.
\newblock {\em Numerical Methods for Singularly Perturbed Differential
  Equations}.
\newblock Studies in Mathematics and its Applications. Springer, 1996.

\bibitem{samarskii_book}
A.A. Samarskii.
\newblock {\em Theory of Difference Schemes}.
\newblock Nauka, Moscow, 1977.

\bibitem{sam_monotone}
A.A. Samarskii, P.P. Matus, V.I. Mazhukin, and I.E. Mozolevski.
\newblock Monotone difference schemes for equations with mixed derivatives.
\newblock {\em Computers \& mathematics with applications}, 44(3-4):501--510,
  2002.

\bibitem{sam_transport}
A.A. Samarskii and P.N. Vabishchevich.
\newblock Monotone difference schemes for the transport equation.
\newblock {\em Doklady Academii Nauk, Russia}, 361(1):21--23, 1998.

\bibitem{sam_triangular}
A.A. Samarskii and P.N. Vabishchevich.
\newblock Monotone difference schemes on triangular grids.
\newblock {\em Doklady Academii Nauk, Russia}, 371(6):742--746, 2000.

\bibitem{sangalli_04}
G.~Sangalli.
\newblock Analysis of the advection-diffusion operator.
\newblock {\em Numer. Math.}, 97(5):779--796, 2004.

\bibitem{S-Gummel}
D.~Scharfetter and H.~Gummel.
\newblock Large-signal analysis of a silicon read diod oscilator.
\newblock {\em IEEE Trans. Electron Devices}, ED-16(205):959--962, 1969.

\bibitem{ASchatz_1974}
A.~H. Schatz.
\newblock An observation concerning {R}itz-{G}alerkin methods with indefinite
  bilinear forms.
\newblock {\em Math. Comp.}, 28(205):952--962, 1974.

\bibitem{tabata_77}
M.~Tabata.
\newblock A finite element approximation corresponding to upwind finite
  differencing.
\newblock {\em mem. Numer. math..}, 4:47--63, 1977.

\bibitem{zikatanov_xu_99}
J.~Xu and L.~Zikatanov.
\newblock A monotone finite element scheme for convection-diffusion equations.
\newblock {\em Math. Comp.}, 68(228):1429--1446, 1999.

\bibitem{JXu_1994a}
Jinchao Xu.
\newblock Two-grid discretization techniques for linear and nonlinear {PDE}s.
\newblock {\em SIAM J. Numer. Anal.}, 33(5):1759--1777, 1996.

\bibitem{Ludmil_1992}
L.~T. Zikatanov.
\newblock A modified {G}alerkin-{P}etrov method for modeling semiconductor
  devices on the basis of the finite element method.
\newblock {\em Mat. Model.}, 4(5):85--99, 1992.
\newblock In Russian.

\bibitem{Ludmil_1991}
L.~T. Zikatanov and M.~S. Kaschiev.
\newblock Finite element method for semiconductor device modeling.
\newblock Technical Report R11-91-371, Communications of the Joint Institute
  for Nuclear Research, Dubna, 1991.
\newblock In Russian.

\end{thebibliography}
\end{document}